\documentclass[10pt,reqno]{amsart}
\usepackage{amsmath}
\usepackage{amssymb}
\usepackage{amsthm}
\usepackage{eepic,epic}
\usepackage{color}
\usepackage{epsfig}
\usepackage{graphicx}
\newlength{\defbaselineskip}
\newcommand{\setlinespacing}[1]%
           {\setlength{\baselineskip}{#1 \defbaselineskip}}

\numberwithin{equation}{section}

\newtheorem{thm}{Theorem}[section]

\newtheorem{lem}[thm]{Lemma}

\theoremstyle{definition}

\theoremstyle{remark}
\newtheorem{rem}[thm]{Remark}
\numberwithin{equation}{section}

\begin{document}

\title[Partial regularity problem]
{On local well-posedness of nonlinear dispersive equations with partially regular data}

\author{Youngwoo Koh, Yoonjung Lee and Ihyeok Seo}

\thanks{This research was supported by NRF-2022R1F1A1061968 (Y. Koh), NRF-2021R1A4A1032418 (Y. Lee), and NRF-2022R1A2C1011312 (I. Seo).}


\address{Department of Mathematics Education, Kongju National University, Kongju 32588, Republic of Korea}
\email{ywkoh@kongju.ac.kr}

\address{Department of Mathematics, Pusan National University, Busan 46241, Republic of Korea}
\email{yjglee@pusan.ac.kr}

\address{Department of Mathematics, Sungkyunkwan University, Suwon 16419, Republic of Korea}
\email{ihseo@skku.edu}

\begin{abstract}
We revisit the local well-posedness theory of nonlinear Schr\"odinger and wave equations in Sobolev spaces $H^s$ and $\dot{H}^s$, $0< s\leq 1$.
The theory has been well established over the past few decades under Sobolev initial data regular with respect to all spatial variables. 
But here, we reveal that the initial data do not need to have complete regularity like Sobolev spaces, but only partially regularity with respect to some variables is sufficient.
To develop such a new theory, we suggest a refined Strichartz estimate which has a different norm for each spatial variable. This makes it possible to extract a different integrability/regularity of the data from each variable. 
\end{abstract}

\maketitle

\section{Introduction}
In this paper we develop a well-posedness theory of nonlinear dispersive equations
with just partially regular initial data.
As representative models, we shall deal with the nonlinear Schr\"odinger and wave equations,
\begin{equation}\label{NLS}
\begin{cases}
i\partial_t u +\Delta u= F_p(u), \\
u(0, x)=f(x),
\end{cases}
\end{equation}
and 
\begin{equation}\label{NLW}
\begin{cases}
\partial_t^2 u-\Delta u=F_p(u), \\
u(0, x)=f(x),\\
\partial_t u(0, x)=g(x),
\end{cases}
\end{equation}
where $(t, x) \in \mathbb{R} \times \mathbb{R}^N$ and the nonlinearity $F_p \in C^1$ with $p>1$ satisfies
\begin{equation}\label{as}
|F_p(u)|\lesssim |u|^p \quad \text{and} \quad |u||F'_p(u)|\sim |F_p(u)|.
\end{equation}

Typical examples of \eqref{as} are $F_p(u)=\pm |u|^{p-1}u$ and $F_p(u)=\pm |u|^p$ with which
the equations enjoy the scaling invariance;
if $u(t,x)$ is a solution of \eqref{NLS} and \eqref{NLW}, then so is
\begin{equation*}
u_{\delta}(t, x)= \delta^{\frac{2}{p-1}} u(\delta^{\sigma} t, \delta x) \quad (\delta>0)
\end{equation*}
with $\sigma=2$ and $\sigma=1$, respectively.
In addition, the Sobolev norm of the rescaled initial data $f_{\delta}(x)=u_{\delta}(0, x)$ is given in terms of the original $f$ as 
\begin{equation}\label{scaling}
\| f_{\delta}\|_{\dot{H}^s(\mathbb{R}^N)} =\delta^{\frac{2}{p-1}+s-\frac{N}{2}} \|f\|_{ \dot{H}^s(\mathbb{R}^N)}
\end{equation}
which determines
the scale-invariant Sobolev space $\dot{H}^{s_c}$ with the so-called critical Sobolev index 
$$
s_c= \frac{N}{2}-\frac{2}{p-1}.
$$
In this regard, the case $s=s_s$ (alternatively $p=1+\frac{4}{N-2s}$) is referred to as \textit{critical}  while the case $s>s_c$ (alternatively $p<1+\frac{4}{N-2s}$) is called \textit{subcritical}.

Particularly in the subcritical case, one could see that the more one assumes regularity $s$ of initial data, the wider a possible range of $p$ in the nonlinearity.
As such, the local well-posedness of \eqref{NLS} and \eqref{NLW} has been extensively studied in Sobolev spaces $H^s(\mathbb{R}^N)$ and  $\dot{H}^s(\mathbb{R}^N)$ over the past few decades and is well established
(see e.g. \cite{C} for Schr\"odinger case, and \cite{K, LS} for wave case).

However, in this paper, we reveal that the initial data do not need to have complete regularity like $H^s(\mathbb{R}^N)$, but only partially regularity like  $L^2(\mathbb{R}^{N-k}; H^{s} (\mathbb{R}^k))$ is sufficient.
The motivation behind this is as follows: we split the spatial variable  $x\in \mathbb{R}^N$ into two variables as $(x, y)\in \mathbb{R}^{N-k} \times \mathbb{R}^k$, $1\leq k\leq N$, and are inspired by an observation that $L_x^2(\mathbb{R}^{N-k}; \dot{H}_y^{s} (\mathbb{R}^k))$ also has a scale-invariance structure similar to \eqref{scaling},
$$\| f_{\delta}\|_{L_x^2 \dot{H}_y^s} =\delta^{\frac{2}{p-1}+s-\frac{N}{2}} \|f\|_{L_x^2 \dot{H}_y^s},$$
regardless of the value of $k$.
Therefore, the same range of $p$ is guaranteed even if the regularity is given only partially,
by which we can naturally expect a new well-posedness theory with improved regularity assumptions. This is the main contribution of this paper.

It should be also noted that $L^2 (\mathbb{R}^{N-k}; H^s (\mathbb{R}^k))$ is rougher than $H^s(\mathbb{R}^N)$. For instance, if we take
$f(x,y) = \phi_1(x) \phi_2(y)$ with $\phi_1 \in L^2(\mathbb{R}^{N-k})\setminus H^s(\mathbb{R}^{N-k})$ and $\phi_2 \in H^s(\mathbb{R}^{k})$,
then $f \in L^2 (\mathbb{R}^{N-k}; H^s (\mathbb{R}^k))$ but $f \notin H^s (\mathbb{R}^{N})$.
In general, for $1\leq k_1 < k_2 <N$,
$$
H^s(\mathbb{R}^N) \subsetneq L^2 (\mathbb{R}^{N-k_2}; H^s (\mathbb{R}^{k_2})) \subsetneq L^2 (\mathbb{R}^{N-k_1}; H^s (\mathbb{R}^{k_1})).
$$

Our first result is the following local well-posedness theorem for the nonlinear Schr\"odinger equation \eqref{NLS} with partially regular initial data $f \in L_x^2 (\mathbb{R}^{N-2}; H_y^s (\mathbb{R}^2))$, i.e., $\langle \nabla_y \rangle^s f \in L^2(\mathbb{R}^N) $.
Our theorem fully recovers the existing $H^s$ well-posedness theorem (\cite{GV0,GV,Ka,Ka2}), even if the regularity of initial data is imposed only on two variables at least.

\begin{thm}\label{mainThm-S}
Let $N\geq3$ and $0<s\leq 1$. If $1<p<1+\frac{4}{N-2s}$ and
$f \in L_x^2 (\mathbb{R}^{N-2}; H_y^s (\mathbb{R}^2))$, then
there exist a time $T>0$ and a unique solution to \eqref{NLS}
$$
u \in C_t([0, T]; L_x^2 H_y^s ) \cap L_t^q ([0,T];L_x^{r} W_y^{s, \widetilde{r}})
$$
for any pair $(q, r, \widetilde{r}) \in \mathbb{A}_{S}(N, 2)$.
\end{thm}

The class $\mathbb{A}_{S}(N,k)$ in the theorem is generally defined for $1\leq k<N$ as
\begin{equation*}\label{mix_Sch_pair}
\mathbb{A}_{S}(N,k)
= \Big\{ (q,r,\widetilde{r}) ~:~ 2<q\leq\infty, ~
2\leq \widetilde{r} \leq r < \infty, ~ \frac{2}{q} + \frac{N-k}{r} + \frac{k}{\widetilde{r}} =  \frac{N}{2}  \Big\}.
\end{equation*} 	
It is easy to check that the special case when $r=\widetilde{r}$ becomes
the usual class of Schr\"odinger admissible pairs, except for the endpoint $(q,r)=(2,\frac{2N}{N-2})$.\\

For the nonlinear wave equation \eqref{NLW}, the next theorem fully recovers the existing $\dot{H}^s$ well-posedness theorem\footnote{The low regularity problem where $0<s<1/2$ has a different conjectured range of $p$. The other case $s>1/2$ is more suitable under the scaling consideration. See \cite{K,KT,L,LS,T}. } (\cite{LS}), even if the regularity of initial data is imposed only on one variable at least.

\begin{thm}\label{mainThm-W}
Let $N\ge2$, $\frac{1}{2}<s\leq 1$ and $ 1+\frac{2}{N-1}< p < 1+\frac{4}{N-2s}$. 
If $ \langle \partial_y  \rangle^{s-\frac{1}{2}}  f \in \dot{H}^{1/2}(\mathbb{R}^{N-1}\times\mathbb{R})$ and $\langle \partial_y\rangle^{s-\frac{1}{2}} g \in \dot{H}^{-1/2}(\mathbb{R}^{N-1}\times\mathbb{R})$, there exist a time $T>0$ and a unique solution to \eqref{NLW}
	$$
    \langle \partial_y  \rangle^{s-\frac{1}{2}}  u \in C_t([0, T];  \dot{H}_{x, y}^{1/2} )\cap L_t^q ([0,T];L_x^{r} L_y^{ \widetilde{r}})
    $$
for any pair $(q, r, \widetilde{r}) \in \mathbb{A}_{W} (N, 1)$.
\end{thm}

The class $\mathbb{A}_{W}(N,k)$ here is generally defined for $1\leq k<N$ as
\begin{equation*}\label{mix_wav_pair}
\begin{aligned}
\mathbb{A}_{W}(N,k)
=\Big\{ &(q,r,\widetilde{r}) ~:~ 2< q\leq\infty,~ 2\leq \widetilde{r} \leq r < \infty,\\
~ &\qquad\frac{1}{q} + \frac{N-k}{r} + \frac{k}{\widetilde{r}} =  \frac{N-1}{2},~ \frac{2}{q}+\frac{N-k-1}{r}+\frac{k}{\tilde{r}} \leq \frac{N-1}{2} \Big\},
\end{aligned}
\end{equation*} 	
and the special case when $r=\widetilde{r}$ becomes the usual class of wave $1/2$-admissible pairs\footnote{A pair $(q,r)\in[2,\infty]\times[2,\infty)$ is wave $s$-admissible if 
	$\frac{2}{q}+\frac{N-1}{r}\leq \frac{N-1}{2}$ and $\frac{1}{q} + \frac{N}{r} =  \frac{N}{2}-s$.} except for the endpoint case $q=2$.\\

As a powerful tool for dealing with nonlinear dispersive equations,
Strichartz estimates (\cite{St, GV, GV2, LS, KT}) have been intensively studied over the past few decades.
In our context, the Strichartz estimates will be refined to exploit a different integrability/regularity for each spatial variable, and it would make it possible to assume only partially regularity in our well-posedness theory.

\begin{thm}\label{thm1}
Let $N\ge1$, $1\leq k \leq N$, and $\sigma=1,2$. If
    \begin{equation}\label{sc2}
    \frac{2}{q} \leq  (N-2-k+\sigma) \Big(\frac{1}{2}-\frac{1}{r}\Big) + k\Big(\frac{1}{2}-\frac{1}{\widetilde{r}}\Big),
    \quad 2\leq \widetilde{r} \leq r < \infty,\quad 2< q \leq \infty,
    \end{equation} 	
then we have 
    \begin{equation}\label{21}
    \big\| e^{it(-\Delta)^{\sigma/2}}f \big\|_{L_t^q (\mathbb{R}; L_x^{r} (\mathbb{R}^{N-k}; L_y^{\widetilde{r}}(\mathbb{R}^k) ))}
    \lesssim \|f\|_{\dot{H}_{x, y}^s}
    \end{equation}
under the scaling condition
    \begin{equation}\label{sc} \frac{\sigma}{q}=(N-k)\Big(\frac{1}{2}-\frac{1}{r}\Big)+k\Big(\frac{1}{2}-\frac{1}{\widetilde{r}}\Big)-s.
    \end{equation}
\end{thm}

\begin{rem}
	The case $\sigma=2$ and $\sigma=1$ in the theorem correspond to the Schr\"odinger and wave flows, respectively. With this in mind,
	the condition \eqref{sc2} combined with \eqref{sc} implies $\mathbb{A}_{S}(N,k)$ and $\mathbb{A}_{W}(N,k)$ for $s=0$ and $s=1/2$, respectively.
\end{rem}

Particularly when $r=\widetilde{r}$, Theorem \ref{thm1} covers the classical Strichartz estimates except for the endpoint case $q=2$. In fact, we further obtain \eqref{21} for the endpoint case when the first inequality in \eqref{sc2} holds strictly. If not, it can be also obtained by making use of the Keel-Tao's bilinear argument.
(For instance, we refer the reader to  \cite{KLS} in which we  have already obtained it for the Schr\"odinger flow when $N\in 2\mathbb{N}$ and $k=N/2$.)
In an approach using the Littlewood-Paley theorem when proving \eqref{21},
the restriction $r, \tilde{r}<\infty$ also follows.
But when $\sigma=2$, one can find the kernel expression of the Schr\"odinger flow which implies the fixed-time estimates without frequency localization,
and so avoid using the Littlewood-Paley theorem.
Hence the classical Strichartz estimates corresponding to $r=\widetilde{r} =\infty$ could be also covered.
Since all the missing cases mentioned above do not give any improvement to the well-posedness results, we shall omit the details about them.

We close the introduction with a remark.
There are some nonlinear Schr\"odinger equations with partial off-axis variations, and partially regular data are naturally considered in the study of those models. 
Related results were studied in \cite{AAS} (and also more numerically in \cite{AKS}).

\

This paper is organized as follows. 
In Section \ref{se2}, we prove the refined Strichartz estimates (Theorem \ref{thm1}) by utilizing the Littlewood-Paley theorem on mixed norm spaces $L^r(\mathbb{R}^{N-k}; L^{\tilde{r}}(\mathbb{R}^k))$ to their frequency localized version which is obtained from a fixed-time estimate (Lemma \ref{lle}) on the mixed spaces. 
The fixed-time estimate is established in Section \ref{se3} making use of a stationary phase method. 
Finally in Section \ref{se4}, we prove  Theorems \ref{mainThm-S} and \ref{mainThm-W} separately by applying the Strichartz estimates to the well-posedness problems.

Throughout this paper, we use $\mathcal{F}f$ or $\hat{f}$ to denote the Fourier transform of $f$.
We also denote $A \lesssim B$ to mean $A \leq CB$ with unspecified constant $C>0$ which may be different at each occurrence.
For functions $f$, $g$ supported in $\mathbb{R}^{N-k}\times\mathbb{R}^{k}$, $f \ast_y g$ means a convolution with respect to $y$-variable,
that is, $$f \ast_y g (x,y) = \int_{\mathbb{R}^k} f(x, y-y') g(x,y')dy'.$$

\section{Proof of Theorem \ref{thm1}}\label{se2}

Let $\psi : \mathbb{R}^{N} \rightarrow [0, 1]$ be a radial smooth cut-off function supported in $\{(\xi, \eta)\in \mathbb{R}^{N-k} \times \mathbb{R}^{k}\,:\, \frac{1}{2} \leq |(\xi, \eta)| \leq 2 \}$
such that
    $$
    \sum_{j \in \mathbb{Z}} \psi(2^{-j}\xi, 2^{-j}\eta) =1.
    $$
For $j\in \mathbb{Z}$, the Littlewood-Paley operator $P_j$ is then defined by
$$
\widehat{P_j f}(\xi, \eta) = \psi_j(\xi, \eta) \hat{f}(\xi, \eta)
$$
where $\psi_j(\xi, \eta) :=  \psi(2^{-j}\xi, 2^{-j}\eta)$
is supported in $\{ (\xi, \eta) \, :\, 2^{j-1} \leq |(\xi, \eta)| \leq 2^{j+1}\}$.
Now we recall the Littlewood-Paley theorem  (see Corollary 2.4 in \cite{TW}) on mixed Lebesgue spaces $L^{r}_x L_y^{\widetilde{r}}(\mathbb{R}^{N-k} \times \mathbb{R}^k)$; for $1<r, \widetilde{r} <\infty$,
\begin{align*}
\|f\|_{ L_x^{r} L_y^{\widetilde{r}}}
\lesssim \Big\| \sum_{j\in\mathbb{Z}}  \left(| P_{j}  f |^{2} \right)^{1/2} \Big\|_{ L_x^{r} L_y^{\widetilde{r}}}.
\end{align*}

Then to show \eqref{21}, we first apply this to $e^{it(-\Delta)^{\sigma/2}}f$ and
use the Minkowski inequality to conclude 
\begin{align}\label{fre1}
\|e^{it(-\Delta)^{\sigma/2}} f\|_{L_t^{q} L_x^{r} L_y^{\widetilde{r}}}^2
&\lesssim \Big\| \sum_{j\in\mathbb{Z}}  \left(| P_{j} e^{it(-\Delta)^{\sigma/2}} f |^{2} \right)^{1/2} \Big\|_{L_t^{q} L_x^{r} L_y^{\widetilde{r}}} ^2 \nonumber\\
&\lesssim \sum_{j\in\mathbb{Z}} \|e^{it(-\Delta)^{\sigma/2}} P_{j}  f\|_{L_t^q L_x^{r} L_y^{\widetilde{r}}}^2
\end{align}
since $q, r, \widetilde{r} \ge 2$ and $r, \widetilde{r} \neq \infty$.
(Note here that $P_{j}$ is commuted with $e^{it(-\Delta)^{\sigma/2}}$.)
Next we assume for the moment the following frequency localized estimates
	\begin{equation}\label{30}
	\big\| e^{it(-\Delta)^{\sigma/2}} P_j f \big\|_{L_t^q L_x^{r} L_y^{\widetilde{r}}}
	\lesssim 2^{js}\|f\|_{L_{x,y}^2}
	\end{equation}
under the conditions \eqref{sc2} and \eqref{sc}.
If we write $P_j f =P_j \widetilde{P}_j f $ with $\widetilde{P}_j = \sum_{k: |j-k|\leq 1} P_{k}$
and apply \eqref{30} with $f=\widetilde{P}_j f$, then
    \begin{align}\label{fre2}
    \sum_{j\in\mathbb{Z}} \| e^{it(-\Delta)^{\sigma/2}} P_j \widetilde{P}_j f \|_{L_t^q L_x^{r} L_y^{\widetilde{r}}}^2
    \lesssim  \sum_{j\in\mathbb{Z}} 2^{2js}\|\widetilde{P}_j f\|_{L_{x,y}^2}^2
    \lesssim \| f \|_{\dot{H}_{x, y}^s}^2 .
    \end{align}
Combining \eqref{fre1} and \eqref{fre2}, we obtain what we want.

\subsection{Proof of \eqref{30}}
It remains to show the frequency localized estimates \eqref{30}. 
By scaling, we only need to show the case $j=0$,
 \begin{equation}\label{300}
 \bigl\| e^{it(-\Delta)^{\sigma/2}} P_0 f \bigl\|_{L_t^q L_x^{r} L_y^{\widetilde{r}}}
 \lesssim \|f\|_{L_{x,y}^2},
 \end{equation}
under the condition \eqref{sc2}.
Indeed, by the change of variables $2^{-j}\xi \rightarrow \xi$ and $2^{-j}\eta \rightarrow \eta$, 
	\begin{align*}
	e^{it(-\Delta)^{\frac{\sigma}2}} P_j f(x,y)
	&=\frac{1}{(2\pi)^{N}}\int_{\mathbb{R}^{N}} e^{i(2^jx,2^jy)\cdot(\xi, \eta)+ i2^{j\sigma}t |(\xi, \eta)|^{\sigma}} \psi(\xi, \eta)  \hat{f}(2^j\xi, 2^j\eta) 2^{jN} d\eta d\xi \\
	&=e^{i2^{j\sigma}t (-\Delta)^{\sigma/2} } P_0 f_j(2^jx, 2^jy)
	\end{align*}
	where $f_j(x,y):=f(2^{-j}x, 2^{-j}y)$. 
Thus \eqref{300} leads us to 
	\begin{align*}
	\bigl\| e^{it(-\Delta)^{\sigma/2}} P_j f \bigl\|_{L_t^q L_x^{r} L_y^{\widetilde{r}}}
	&=2^{j(-\frac{\sigma}{q}-\frac{N-k}{r}-\frac{k}{\widetilde{r}})}\bigl\|
	e^{it(-\Delta)^{\sigma/2}} P_0 f_j \bigl\|_{L_t^q L_x^{r} L_y^{\widetilde{r}}}\\
	&\lesssim 2^{j(-\frac{\sigma}{q}-\frac{N-k}{r}-\frac{k}{\widetilde{r}})}\|f_j\|_{L_{x,y}^2} \\
	&= 2^{js}\|f\|_{L_{x,y}^2}
	\end{align*}
	under the condition \eqref{sc}, as desired.

Now we shall prove \eqref{300} by using the following lemma which shows 
fixed-time estimates for the propagator $e^{it(-\Delta)^{\sigma/2}}$ on mixed Lebesgue spaces.
We will obtain the lemma in the next section, Section \ref{se3}.

\begin{lem}\label{lle}
Let $N \ge 2$, $1\leq k \leq N$, and $\sigma=1,2$.
Assume that $2 \leq \widetilde{r} \leq r \leq \infty$.
Then for $(x, y)\in \mathbb{R}^{N-k} \times \mathbb{R}^k$
we have
	\begin{equation}\label{d1}
	\| e^{it(-\Delta)^{\sigma/2}} P_0\, g  \|_{L_x^{r} L_y^{\widetilde{r}}}
	\lesssim (1+|t|)^{-\beta_{\sigma}(r, \widetilde{r}) } \| g\|_{L_x^{r'} L_y^{\widetilde{r}'}}
	\end{equation}
where $\beta_{\sigma}(r, \widetilde{r})=(N-k-2+\sigma) (\frac{1}{2}-\frac{1}{r}) + k(\frac{1}{2}-\frac{1}{\widetilde{r}})$.
\end{lem}

By the standard $TT^{\ast}$ argument, \eqref{300} is equivalent to 
	\begin{equation*}\label{TT*}
	\left\| \int_{\mathbb{R}} e^{i(t-\tau) (-\Delta)^{\sigma/2}} P_0\,g \, d\tau \right\|_{L_t^q L_x^{r} L_y^{\widetilde{r}}} \lesssim \left\| g \right\|_{L_t^{q'} L_x^{r'} L_y^{\widetilde{r}'}},
	\end{equation*}
and then by \eqref{d1} 
    \begin{align*}
    \bigg\| \int_{\mathbb{R}} e^{i(t-\tau) (-\Delta)^{\sigma/2}} P_0\,g \, d\tau \bigg\|_{L_t^q L_x^{r} L_y^{\widetilde{r}}}
    & \leq\bigg\| \int_{\mathbb{R}} \big\| e^{i(t-\tau) (-\Delta)^{\sigma/2}}P_0\,g \big\|_{L_x^{r} L_y^{\widetilde{r}}} d\tau \bigg\|_{L_t^q} \\
    &\lesssim \Big\| (1+|\cdot|)^{-\beta_{\sigma}(r, \widetilde{r})} \ast_t \| g\|_{L_x^{r'} L_y^{\widetilde{r}'}} \Big\|_{L_t^q}.
    \end{align*}
Note that the first condition in \eqref{sc2} equals $2/q\leq \beta_{\sigma}(r, \widetilde{r})$.
When $2/q<\beta_{\sigma}(r, \widetilde{r})$ and $2\leq q \leq \infty$,
we apply Young's inequality to get
    \begin{align*}
    \Big\| (1+|\cdot|)^{-\beta_{\sigma}(r, \widetilde{r})} \ast_t \| g\|_{L_x^{r'} L_y^{\widetilde{r}'}} \Big\|_{L_t^q}
    &\lesssim \| (1+|\cdot|)^{-\beta_{\sigma}(r, \widetilde{r})} \|_{L_t^{\frac{q}{2}}} \| g\|_{L_t^{q'}L_x^{r'} L_y^{\widetilde{r}'}} \\
    &\lesssim  \| g\|_{L_t^{q'}L_x^{r'} L_y^{\widetilde{r}'}}.
    \end{align*}
In the borderline $2/q =\beta_{\sigma}(r, \widetilde{r})$ with $2< q < \infty$,
we apply the Hardy-Littlewood-Sobolev inequality (\cite{S2}, VIII, Section 4.2) 
for one-dimension\footnote{$\big\| |\cdot|^{-\alpha} \ast f\big\|_{L^q} \lesssim \|f\|_{L^{p}}$ holds
for $0 < \alpha <1$, $1 \leq p < q < \infty $ and $ 1/q+1=1/p+\alpha$.},
to get
    \begin{align*}
    \Big\| (1+|\cdot|)^{-\beta_{\sigma}(r, \widetilde{r})} \ast_t \| g\|_{L_x^{r'} L_y^{\widetilde{r}'}} \Big\|_{L_t^q}
    &\lesssim \Big\|   |\cdot|^{-\beta_{\sigma}(r, \widetilde{r})} \ast_t \| g\|_{L_x^{r'} L_y^{\widetilde{r}'}} \Big\|_{L_t^q} \\
    &\lesssim \| g\|_{L_t^{q'}L_x^{r'} L_y^{\widetilde{r}'}}.
    \end{align*}
If $q=\infty$, $\beta_{\sigma}(r, \widetilde{r})=0$ and so $r=\widetilde{r}=2$. In this case 
\eqref{d1} clearly holds by Plancherel's theorem.
The proof is now complete.

\section{Fixed-time estimates}\label{se3}
In this section we prove Lemma \ref{lle}.
By the Riesz-Thorin interpolation theorem \cite{BL},
we only need to obtain \eqref{d1} for the following three cases:
	\begin{itemize}
		\item[(a)] $r=\widetilde{r}=2$,
		\item[(b)] $r=\widetilde{r}=\infty$,
		\item[(c)] $r=\infty$ and $\widetilde{r}=2$.
	\end{itemize}

The case $(a)$ clearly follows from Planchrel's theorem. 
For the case $(b)$ we first write
	\begin{align*}
	e^{it (-\Delta)^{\sigma/2}} P_0\,g(x,y)
	&= \int_{\mathbb{R}^{N-k}} \int_{\mathbb{R}^{k}} K_{ \sigma}(x-x', y-y',t) g(x', y') dy' dx' \\
	&= K_{\sigma} \ast_{x, y} g,
	\end{align*}
	where
	\begin{equation}\label{ker}
	K_{\sigma}(x,y,t)
    = \frac{1}{(2\pi)^{N}} \int_{\mathbb{R}^{N-k}} \int_{\mathbb{R}^{k}} e^{i(x,y)\cdot(\xi,\eta)} e^{it|(\xi,\eta)|^{\sigma}} \psi(\xi,\eta)\, d\eta d\xi,
	\end{equation}
and use Young's inequality to see
	\begin{align}\label{b}
	\bigl\|e^{it (-\Delta)^{\sigma/2}} P_0\,g\bigl\|_{L_{x,y}^{\infty}}
	=\bigl\| K_{ \sigma} \ast_{x, y} g \bigl\|_{L_{x,y}^{\infty}}
	\leq \| K_{ \sigma}\|_{L_{x,y}^{\infty}} \| g \|_{L_{x,y}^{1}}.
	\end{align}
To bound the kernel $K_{\sigma}$ here,
we shall make use of the following stationary phase method (see \cite{S2}, VIII, Section 5, B).

	\begin{lem}\label{lem4}
		Let $H$ be the Hessian matrix given by $(\frac{\partial^2 }{\partial \xi_i \partial \xi_j})$.
		Suppose that $\psi$ is a compactly supported smooth function on $\mathbb{R}^N$ and $\phi$ is a smooth function satisfying rank $H\phi \ge k$ on the support of $\psi$.
		Then, for $(x, t) \in \mathbb{R}^{N+1}$
		\begin{equation*}\label{sta}
		\left| \int e^{i(x, t) \cdot (\xi, \phi(\xi))} \psi(\xi) d\xi \right| \leq C (1+|(x, t)|)^{-k/2}.
		\end{equation*}
	\end{lem}
	
	Indeed, take a function $\phi(\xi, \eta)=|(\xi, \eta)|^{\sigma}$ defined on 
	$\mathbb{R}^{N-k} \times \mathbb{R}^k$, and note that 
	 rank $H\phi \ge N-1$ for $\sigma=1$ and rank $H\phi \ge N$ for $\sigma=2$,
	 on the support of $\psi$ in \eqref{ker} which is 
	 $\{(\xi, \eta)\in \mathbb{R}^{N-k} \times \mathbb{R}^k :\, 1/2<|(\xi, \eta)| < 2\}$.
	Now apply the lemma with $k=N-2+\sigma$ to \eqref{ker} as
	\begin{align}\label{esk}
	|K_{\sigma}(x, y, t)|
	&= \frac{1}{(2\pi)^{N}}\left| \int_{\mathbb{R}^{N-k}} \int_{\mathbb{R}^{k}} e^{i(x,y, t)\cdot(\xi,\eta, |(\xi, \eta)|^{\sigma})} \psi(\xi,\eta)\, d\eta d\xi \right| \nonumber \\
	&\lesssim (1+|(x,y,t)|)^{-\frac{N-2+\sigma}{2}}\nonumber\\
	&\lesssim (1+|t|)^{-\frac{N-2+\sigma}{2}}.
	\end{align}
Notice here that $\beta_{\sigma}(\infty, \infty)=(N-2+\sigma)/2\geq0$.
Consequently, by \eqref{b} and \eqref{esk}, we get the case $(b)$.
	
For the case $(c)$, we use Minkowski's inequality and Plancherel's theorem with respect to $y$ to see
    \begin{align}\label{3a}
    \bigl\|e^{it (-\Delta)^{\sigma/2}} P_0\,g\bigl\|_{L_x^{\infty} L_y^2}
    &= \big\| \| K_{ \sigma} \ast_{x,y} g \|_{L_y^2} \big\|_{L_x^{\infty}}\nonumber\\
    &\leq  \bigg\|  \int_{\mathbb{R}^{N-k}} \big\| K_{ \sigma} (x-x',\cdot) \ast_y g(x',\cdot) \big\|_{L_y^2} \, dx' \bigg\|_{L_x^{\infty}}\nonumber\\
    &=  \bigg\|  \int_{\mathbb{R}^{N-k}} \big\| \widetilde{K}_{\sigma} (x-x',\cdot) \tilde{g}(x',\cdot) \big\|_{L_{\eta}^2} \, dx' \bigg\|_{L_x^{\infty}}.
    \end{align}
Here, $\tilde{g} =\mathcal{F}_y ( g(x, \cdot))$ denotes the spatial Fourier transform in the variable $y \in \mathbb{R}^k$ and similarly
	\begin{align*}
	\widetilde{K}_{\sigma}(x, \eta,t)
	=\frac{1}{(2\pi)^{N}}\int_{\mathbb{R}^{N-k}}e^{ix\cdot\xi} e^{it|(\xi,\eta)|^{\sigma}} \psi(\xi,\eta)   d\xi.
	\end{align*}

Applying Lemma \ref{lem4} with phase functions 
$\phi_{\sigma,\eta} (\xi)=|(\xi, \eta)|^{\sigma}$ for fixed $|\eta| \leq 2$, we now get
	\begin{equation}\label{3b}
	\sup_{|\eta|\leq2}| \widetilde{K}_{\sigma}(x, \eta, t) |\lesssim (1+|(x, t)|)^{-\frac{N-k-2+\sigma}{2}};
	\end{equation}
first calculate the $(N-k)\times(N-k)$ Hessian matrix of $\phi_{\sigma,\eta}$ as
    $$
    H \phi_{2,\eta} (\xi) =
    \begin{pmatrix}
    2    &   0 &   0 \\
    0    &   \ddots    &   0 \\
    0    &   0   &  2
    \end{pmatrix} 
    $$
and
    $$
    H \phi_{1,\eta} (\xi) = \frac{1}{(|\xi|^2 +|\eta|^2)^{3/2}}
    \begin{pmatrix}
    |\xi|^2 + |\eta|^2 - \xi_1^2   &   \cdots  &   -\xi_1 \xi_{N-k} \\
    \vdots     &   \ddots    &   \vdots \\
    -\xi_1 \xi_{N-k}    &   \cdots &   |\xi|^2 + |\eta|^2 - \xi_{N-k}^2
    \end{pmatrix} .
    $$
Then clearly rank $H \phi_{2,\eta} (\xi) = N-k$. For the latter case, consider 
    $$
    H \phi_{1,\eta} (\xi_1,\cdots,0) = \frac{1}{(\xi_1^2 +|\eta|^2)^{3/2}}
    \begin{pmatrix}
    |\eta|^2   &   \cdots  &   0 \\
    \vdots     &   \ddots    &   \vdots \\
    0    &   \cdots &   \xi_1^2 + |\eta|^2
    \end{pmatrix}
    $$
since the phase function $\phi_{1,\eta} (\xi)$ is rotational invariant,
and conclude rank $H \phi_{1,\eta} (\xi) \geq N-k-1$ for all $|\eta| \leq 2$
since at least one of $|\xi_1|$ and $|\eta|$ is away from zero under $1/2 \leq |(\xi,\eta)| \leq 2$.

Now by \eqref{3b} and Young's inequality, we bound the right-hand side of \eqref{3a} as
	\begin{align}\label{3c}
	\bigg\|\,  \int_{\mathbb{R}^{N-k}} \| \widetilde{K}_{\sigma} (x-x',\cdot) \tilde{g}(x',\cdot) \|_{L_{\eta}^2} \, dx' \bigg\|_{L_x^{\infty}}
	& \lesssim \bigg\|\,  (1+|(\cdot, t)|)^{-\frac{N-k-2+\sigma}{2}} \ast_x \| \tilde{g} \|_{L_{\eta}^2}  \bigg\|_{L_x^{\infty}} \nonumber\\
	&\lesssim  (1+|t|)^{-\frac{N-k-2+\sigma}{2}} \big\| \| \tilde{g} \|_{L_{\eta}^{2}} \big\|_{L_x^{1}} \nonumber\\
	&\leq (1+|t|)^{-\frac{N-k-2+\sigma}{2}}  \|g \|_{L_x^{1}L_y^2}.
	\end{align}
Note here that  $\beta_{\sigma}(\infty, 2)=\frac{N-k-2+\sigma}{2}\geq0$.
Combining \eqref{3a} and \eqref{3c} finally gives the case $(c)$,
	\begin{equation*}
	\bigl\|e^{it (-\Delta)^{\sigma/2}} P_0\,g\bigl\|_{L_x^{\infty} L_y^2}
	\lesssim  (1+|t|)^{-\beta_{\sigma}(\infty, 2)}  \|g \|_{L_x^{1}L_y^2}.
	\end{equation*}

\section{Well-posedness}\label{se4}
In this section we prove Theorems \ref{mainThm-S} and \ref{mainThm-W} by making use of the refined Strichartz estimates \eqref{21}.

We shall first mention some simple facts about the nonlinearity $F_p(u)$ to be used in the proof:
\begin{equation}\label{nonli}
|F_p(u) - F_p(v) |\leq C \left( |u|^{p-1} + |v|^{p-1} \right) |u-v|
\end{equation}
and 
\begin{equation}\label{frac_chainrule}
	\| | \nabla |^s F_p(u)\|_{L^{c}} \leq C\| u\|_{L^{a(p-1)}}^{p-1} \| |\nabla |^s u\|_{L^{b}}
\end{equation}
where $s\in (0,1]$, $1<a, b, c <\infty$ and $1/c =1/a +1/b$.
Using \eqref{as}, the former is shown by
\begin{align*}
	F_p(u) - F_p(v) 
	& =  \int_0^1 \frac{d}{d\tau} F_p (\tau u + (1-\tau)v) d\tau \nonumber\\
	& =  \int_0^1 (u-v) \cdot  F'_p (\tau u + (1-\tau)v) d\tau,
\end{align*}
while the latter is an immediate consequence of the following fractional chain rule
(see \cite{Ka2}, Lemma A2).

\begin{lem}\label{fractional chain rule}
	Let $F \in C^1 (\mathbb{C};\mathbb{C})$ with $F(0)=0$ and $|F'(z)|\leq|z|^{k-1}$
	for $z\in\mathbb{C}$ and $k\geq1$.	
	If $0\leq s\leq1$, then
	$$
	\|| \nabla |^{s} F(u)\|_{L^c} \leq C \|F'(u)\|_{L^a} \|| \nabla |^{s}u\|_{L^b}
	$$
for $1<a, b, c<\infty$ and $1/c=1/a+1/b$. 
\end{lem}

\subsection{Schr\"odinger case}
Now we prove Theorem \ref{mainThm-S} based on the contraction mapping principle.
By Duhamel's principle, the solution map of \eqref{NLS} is given as
\begin{equation*}
	\Phi(u)=e^{it \Delta} f -i \int_0^t e^{i(t-\tau)\Delta } F_p(u)\, d\tau.
\end{equation*}
For $0<s \leq 1$ and suitable values of $T, A>0$, it suffices to prove that $\Phi$ defines a contraction map on
\begin{align*}
	X(T, A)=  \big\lbrace  u \in C([0,T]; L^2_x H^s_y)  \cap L_t^q ([0,T] &;L_x^{r} W_y^{s, \widetilde{r}}):\\
	 &\sup_{(q, r, \widetilde{r}) \in \mathbb{A}_{S}(N,2)}\| u \|_{L_t^{q}([0,T]; L_x^r W_y^{s,\widetilde{r}})} \leq A \big\rbrace
\end{align*}
equipped with the distance
$$d(u, v)= \sup_{(q, r, \widetilde{r}) \in \mathbb{A}_{S}(N,2)} \|u-v\|_{L_t^{q}(I; L_x^r L_y^{\widetilde{r}})}.$$

For the proof some preparation is needed.
Let $x\in\mathbb{R}^{N-2}$ and $y\in\mathbb{R}^{2}$.
For any pairs $(q,r,\widetilde{r}) \in \mathbb{A}_{S}(N,2)$, it then follows from
\eqref{21} with $\sigma=2$, $k=2$ and $s=0$ that
\begin{equation}\label{hstr_Sch}
\|e^{it\Delta} f\|_{L_t^q (I ; L_x^rL_y^{\widetilde{r}})}
\leq C \|f\|_{L^2(\mathbb{R}^N)}
\end{equation}
for any time interval $I=[0, T]$.
Furthermore, this yields the inhomogeneous estimates
\begin{equation}\label{istr_Sch}
\left\|\int_0^t e^{i(t-\tau)\Delta} F_p(u)\, d\tau \right\|_{L_t^q (I ; L_x^rL_y^{\widetilde{r}} )}
\leq C \|F_p\|_{L_t^{a'} (I ; L_x^{b'}L_y^{\widetilde{b}'})}
\end{equation}
for any pairs $(q,r,\widetilde{r}), (a,b,\widetilde{b})  \in \mathbb{A}_{S}(N,2)$
by the standard $TT^{\ast}$ argument and the Christ-Kiselev Lemma \cite{CK}.
We also need the following nonlinear estimates.
\begin{lem}\label{lem1}
Let $N\geq 3$, $0< s \leq 1$ and $1< p<1+\frac{4}{N-2s}$.
For any time interval $I=[0,T]$, there are a pair $(q_0,r_0,\widetilde{r}_0) \in \mathbb{A}_{S}(N,2)$ and $\beta_0(N,p,s)>0$ such that
\begin{equation}\label{no1-s}
\| \langle \nabla_y \rangle^s F_p(u)\|_{L_t^{q_0'} (I; L_x^{r_0'} L_y^{\widetilde{r}_0'}) }
\leq C T^{\beta_0(N,p,s) } \| u\|_{L_t^{q_0}(I;L_x^{r_0} W_y^{s,\widetilde{r}_0})}^p
\end{equation}
and
\begin{equation}\label{no2-s}
\begin{aligned}
&\| F_p(u)-F_p(v)\|_{L_t^{q_0'} (I; L_x^{r_0'} L_y^{\widetilde{r}_0'})} \\
&\leq C T^{\beta_0(N,p,s) } \Big( \| u\|_{L_t^{q_0}(I;L_x^{r_0} W_y^{s,\widetilde{r}_0})}^{p-1}  +\| v\|_{L_t^{q_0}(I;L_x^{r_0} W_y^{s,\widetilde{r}_0})}^{p-1} \Big)
\|u-v\|_{L_t^{q_0}(I;L_x^{r_0} L_y^{\widetilde{r}_0})} .
\end{aligned}
\end{equation}
\end{lem}

\begin{proof}
Let $N\geq 3$ and $0< s \leq 1$.
For fixed $1< p<1+\frac{4}{N-2s}$, we take a small $\epsilon_0>0$ such that
$$
\frac{(1-s)(p-1)}{2} < \epsilon_0 < \min\Big\{ \frac{N-2}{4}\Big( 1+\frac{4}{N-2} -p \Big),~ \frac{p-1}{2} \Big\},
$$
and then take a pair $(q_0, r_0, \widetilde{r}_0)$ as
\begin{equation}\label{holl}
\frac{1}{r_0}= \frac{1}{p+1}, \quad
\frac{1}{\widetilde{r}_0}= \frac{1}{2} -\frac{\epsilon_0}{p+1}, \quad
\frac{1}{q_0}= \frac{N-2}{4} - \frac{N-2}{2(p+1)} +\frac{\epsilon_0}{p+1}
\end{equation}
so that $(q_0, r_0, \widetilde{r}_0) \in \mathbb{A}_{S}(N,2)$.

Now we note that
\begin{equation*}
	\|\langle \nabla_y \rangle^s F_p(u)\|_{L_t^{q_0'} (I; L_x^{r_0'} L_y^{\widetilde{r}_0'}) }
	\lesssim \|F_p(u)\|_{L_t^{q_0'} (I; L_x^{r_0'} L_y^{\widetilde{r}_0'}) } +\||\nabla_y|^s F_p(u)\|_{L_t^{q_0'} (I; L_x^{r_0'} L_y^{\widetilde{r}_0'}) }. 
\end{equation*}
By \eqref{as} and H\"older's inequality with the first two ones in \eqref{holl}, the first term in the right side above is bounded as  
\begin{align*}
	\|F_p(u)\|_{L_t^{q_0'} (I; L_x^{r_0'} L_y^{\widetilde{r}_0'}) } 
	&\lesssim \|\chi_I(t)|u|^{p-1}|u|\|_{L_t^{q_0'} (I; L_x^{r_0'} L_y^{\widetilde{r}_0'}) } \\
	&\leq C T^{1-\frac{p+1}{q_0}} \|u\|_{L_t^{q_0}(I;L_x^{r_0} L_y^{(p-1)(p+1)/2\epsilon_0})}^{p-1}\|u\|_{L_t^{q_0}(I;L_x^{r_0}L_y^{\widetilde{r}_0})}.
\end{align*}
For the second term, we use \eqref{frac_chainrule} and then apply the H\"older inequality as before to get
$$
\||\nabla_y|^s F_p(u)\|_{L_t^{q_0'} (I; L_x^{r_0'} L_y^{\widetilde{r}_0'}) }
\leq C T^{1-\frac{p+1}{q_0}} \|u\|_{L_t^{q_0}(I;L_x^{r_0} L_y^{(p-1)(p+1)/2\epsilon_0})}^{p-1} \|u\|_{L_t^{q_0}(I;L_x^{r_0}W_y^{s, \widetilde{r}_0})}.
$$
Since $ W^{s, \widetilde{r}_0} \subseteq L^{\widetilde{r}_0}$ for $s\ge0$,
and $$W^{s, \widetilde{r}_0} \hookrightarrow W^{1-\frac{2\epsilon_0}{p-1}, \widetilde{r}_0} \hookrightarrow L^{\frac{(p-1)(p+1)}{2\epsilon_0}}$$
by the Sobolev embedding in two dimensions,
we finally arrive at 
\begin{equation*}
	\|\langle \nabla_y \rangle^s F_p(u)\|_{L_t^{q_0'} (I; L_x^{r_0'} L_y^{\widetilde{r}_0'}) }
	\leq C T^{\beta_0(N,p,s)} \|u\|_{L_t^{q_0}(I;L_x^{r_0} W_y^{s, \widetilde{r}_0})}^p 
\end{equation*}
with $\beta_0(N,p,s) = 1-\frac{p+1}{q_0}$.
Note here that 
\begin{align*}
	\beta_0(N,p,s)
	= 1-\frac{p+1}{q_0}
	= \frac{N-2}{4} \big(1+\frac{4}{N-2} -p \big) -\epsilon_0 >0
\end{align*}
from the last one in \eqref{holl} and the choice of $\epsilon_0$.

The estimate \eqref{no2-s} is similarly but more easily obtained by using \eqref{nonli},  H\"older's inequality, and then the embedding. 
\end{proof}

Let us now prove that $\Phi$ is a contraction on $X(T,A)$.  
We first show that $\Phi(u) \in X$ for $u\in X$.
By Plancherel's theorem and a dual version of \eqref{hstr_Sch}, we see 
\begin{align*}
\sup_{t\in I}\|\Phi(u)\|_{L_x^2 H_y^s}
&\leq C\|f\|_{L_x^2 H_y^s}+C \sup_{t\in I} \bigg\| \langle \nabla_y \rangle^s  \int_0^t e^{-i\tau\Delta} F_p(u) \, d\tau\bigg\|_{L_x^2 L_y^2} \nonumber\\
&\leq C\|f\|_{L_x^2 H_y^s}+C\| \langle \nabla_y \rangle^s F_p(u)\|_{L_t^{q_0'}(I; L_x^{r_0'} L_y^{\widetilde{r}_0'})}
\end{align*}
for some $(q_0,r_0,\widetilde{r}_0)\in \mathbb{A}_S(N,2)$ for which Lemma \ref{lem1} holds.
By \eqref{no1-s} in Lemma \ref{lem1}, we then conclude that
    \begin{align*}
    \sup_{t\in I}\|\Phi(u)\|_{L_x^2 H_y^s}
    \leq C\|f\|_{L_x^2 H_y^s}+ CT^{\beta_0(N,p,s)} A^{p}
    \end{align*}
for $u\in X$.
Similarly, using \eqref{hstr_Sch}, \eqref{istr_Sch} and \eqref{no1-s} in Lemma \ref{lem1},
we also see
    \begin{align*}
    \|\Phi(u)\|_{L^{q}(I; L_x^{r} W_y^{s, \widetilde{r}})}
    &\leq C\|f\|_{L_x^2 H_y^s} +C\| \langle \nabla_y \rangle^s F_p(u)\|_{L^{q_0'}(I; L_x^{r_0'} L_y^{\widetilde{r}_0'}) } \nonumber\\
    &\leq C\|f\|_{L_x^2 H_y^s} +CT^{\beta_0(N,p,s)} A^{p}
    \end{align*}
for $u\in X$.
Hence, $\Phi(u)\in X$ if
\begin{equation}\label{ne}
C\|f\|_{L_x^2 H_y^s}+CT^{\beta_0(N,p,s)} A^{p} \leq A.
\end{equation}

On the other hand, by using \eqref{istr_Sch} and  \eqref{no2-s} in Lemma \ref{lem1},
    \begin{align*}\label{33}
    d\left( \Phi(u), \Phi(v) \right)
    &= d\left(\int_0^t e^{i(t-\tau)\Delta} F_p(u)\, d\tau ,\int_0^t e^{i(t-\tau)\Delta} F_p(v)\, d\tau  \right) \nonumber\\
    & \leq  C\| F_p(u)-F_p(v)\|_{L^{q_0'}(I; L_x^{r_0'} L_y^{\widetilde{r}_0'}) }\nonumber\\
    &\leq 2CT^{\beta_0(N,p,s)} A^{p-1} \, d(u, v)
    \end{align*}
for $u,v\in X$.
By choosing $A=2C\|f\|_{L_x^2 H_y^s}$ and taking $T$ sufficiently small so that
$$
2CT^{\beta_0(N,p,s)} A^{p-1} \leq \frac{1}{2},
$$
we see that \eqref{ne} holds and 
    $$
    d\left( \Phi(u), \Phi(v) \right) \leq \frac{1}{2} d(u, v).
    $$

\subsection{Wave case}\label{se5}
Here we prove Theorem \ref{mainThm-W} similarly as above. 
For simplicity, we prove the theorem by substituting $s-1/2$ with $s$.
Let $x\in\mathbb{R}^{N-1}$ and $y\in\mathbb{R}$. Instead of \eqref{hstr_Sch} and \eqref{istr_Sch}, we will this time use 
\begin{equation}\label{hstr_Wav}
	\|e^{it\sqrt{-\Delta}} f\|_{L_t^q (I ; L_x^rL_y^{\widetilde{r}})}
	\leq C \|f\|_{\dot{H}^{1/2}(\mathbb{R}^N)}
\end{equation}
and
\begin{equation}\label{istr_Wav}
	\left\| |\nabla|^{-1} \int_0^t e^{i(t-\tau)\sqrt{-\Delta}}F_p(u) \, d\tau \right\|_{L_t^q (I ; L_x^rL_y^{\widetilde{r}})}
	\leq C \|F_p\|_{L_t^{a'} (I ; L_x^{b'}L_y^{\widetilde{b}'} )}
\end{equation}
for any pairs $(q,r,\widetilde{r}), (a,b,\widetilde{b}) \in \mathbb{A}_{W}(N,1)$.
Estimate \eqref{hstr_Wav} follows immediately from Theorem \ref{thm1} with 
 $\sigma=1$, $k=1$ and $s=1/2$, and it yields \eqref{istr_Wav}
by the standard $TT^{\ast}$ argument and the Christ-Kiselev Lemma \cite{CK}.

Additionally, we need the following nonlinear estimates (cf. Lemma \ref{lem1}):

\begin{lem}\label{lem2}
Let $N\geq 2$, $0<s \leq 1/2$ and $1+\frac{2}{N-1}< p<1+\frac{4}{N-1-2s}$.
For any time interval $I=[0,T]$, there are a pair $(q_1,r_1,\widetilde{r}_1) \in \mathbb{A}_{W}(N,1)$ and $\beta_1(N,p,s)>0$ such that
	\begin{equation}\label{no1-w}
	\| \langle \partial_y \rangle^s F_p(u)\|_{L^{q_1'} (I; L_x^{r_1'} L_y^{\widetilde{r}_1'}) }
	\leq C T^{\beta_1(N,p,s) } \| u\|_{L^{q_1}(I;L_x^{r_1} W_y^{s,\widetilde{r}_1})}^p
	\end{equation}
and
    \begin{equation}\label{no2-w}
    \begin{aligned}
    &\| F_p(u)-F_p(v)\|_{L^{q_1'} (I; L_x^{r_1'} L_y^{\widetilde{r}_1'})} \\
    &\leq C T^{\beta_1(N,p,s) } \Big( \| u\|_{L^{q_1}(I;L_x^{r_1} W_y^{s,\widetilde{r}_1})}^{p-1}
    +\| v\|_{L^{q_1}(I;L_x^{r_1} W_y^{s,\widetilde{r}_1})}^{p-1} \Big)
    \|u-v\|_{L^{q_1}(I;L_x^{r_1} L_y^{\widetilde{r}_1})} .
    \end{aligned}
    \end{equation}
\end{lem}

\begin{proof}
	The proof is similar as in Lemma \ref{lem1}.
	Let $N\geq2$ and $0<s\leq1/2$. 
	For fixed $1+\frac{2}{N-1}< p<1+\frac{4}{N-1-2s}$,
	we take a small $\epsilon_1>0$ such that
	\begin{align*}
		\max \left\{  \frac{(1-2s)(p-1)}{2}, 1-\frac{(N-2)(p-1)}{2}  \right\}<  \epsilon_1
		< \min \left\{ \frac{p-1}{2},  2-\frac{(N-2)(p-1)}{2}\right\},
	\end{align*}
	ant then take a pair $(q_1, r_1, \widetilde{r}_1)$ as
	\begin{equation}\label{holl2}
	\frac{1}{r_1}= \frac{1}{p+1}, \quad
	\frac{1}{\widetilde{r}_1}= \frac{1}{2} -\frac{\epsilon_1}{p+1}, \quad
	\frac{1}{q_1}= \frac{N-2}{2} - \frac{N-1}{p+1} +\frac{\epsilon_1}{p+1}
	\end{equation}
	so that
	 $(q_1, r_1, \widetilde{r}_1) \in \mathbb{A}_{W}(N,1)$.
	
Now we note that
	\begin{align*}
		\|\langle \partial_y \rangle^s F_p(u)\|_{L_t^{q_1'} (I; L_x^{r_1'} L_y^{\widetilde{r}_1'}) }
		&\leq \|F_p(u)\|_{L_t^{q_1'} (I; L_x^{r_1'} L_y^{\widetilde{r}_1'}) } +\||\partial_y|^s F_p(u)\|_{L_t^{q_1'} (I; L_x^{r_1'} L_y^{\widetilde{r}_1'}) }. 
	\end{align*}
	By \eqref{as} and H\"older's inequality with the first two ones in \eqref{holl2}, the first term in the right side above is bounded as
	\begin{align*}
		\|F_p(u)\|_{L_t^{q_1'} (I; L_x^{r_1'} L_y^{\widetilde{r}_1'}) }
		&\leq C T^{1-\frac{p+1}{q_1}} \|u\|_{L_t^{q_1}(I;L_x^{r_1} L_y^{(p-1)(p+1)/2\epsilon_1})}^{p-1} \|u\|_{L_t^{q_1}(I;L_x^{r_1}L_y^{\widetilde{r}_1})}.
	\end{align*}
For the second term, we use \eqref{frac_chainrule} and then apply the H\"older inequality as before to get
$$
\||\partial_y|^s F_p(u)\|_{L_t^{q_1'} (I; L_x^{r_1'} L_y^{\widetilde{r}_1'}) } \leq C T^{1-\frac{p+1}{q_1}} \|u\|_{L_t^{q_1}(I;L_x^{r_1} L_y^{(p-1)(p+1)/2\epsilon_1})}^{p-1} \|u\|_{L_t^{q_1}(I;L_x^{r_1}W_y^{s, \widetilde{r}_1})}.
$$
Since $ W^{s, \widetilde{r}_1} \subseteq L^{\widetilde{r}_1}$ for $s\ge0$, and
	$$
	W^{s, \widetilde{r}_1} \hookrightarrow W^{\frac{1}{2}-\frac{\epsilon_1}{p-1}, \widetilde{r}_1} \hookrightarrow L^{\frac{(p-1)(p+1)}{2\epsilon_1}}
	$$
	by the Sobolev embedding in one dimension, we finally arrive at 
	\begin{align*}
		\|\langle \partial_y \rangle^s F_p(u)\|_{L^{q_1'} (I; L_x^{r_1'} L_y^{\widetilde{r}_1'}) }
		\leq C T^{\beta_1(N,p,s)} \|u\|_{L^{q_1}(I;L_x^{r_1} W_y^{s, \widetilde{r}_1})}^p
	\end{align*}
	with $\beta_1(N,p,s)= 1-\frac{p+1}{q_1}$.
	Note here that 
	$$
	\begin{aligned}
		\beta_1(N,p,s)
		=1-\frac{p+1}{q_1}
		= \frac{N-2}{2} \big(1+\frac{4}{N-2} -p \big) -\epsilon_1 >0
	\end{aligned}
	$$
	from the last one in \eqref{holl2} and the choice of $\epsilon_1$.
	
	Using \eqref{nonli}, H\"older's inequality, and then the embedding,  
	the estimate \eqref{no2-w} follows similarly but more easily.
\end{proof}

Now we are ready to prove Theorem \ref{mainThm-W} based on the contraction mapping principle.
By Duhamel's principle, the solution map of \eqref{NLW} is given as
    \begin{equation*}
    \Phi(u)=\cos (t\sqrt{-\Delta}) f + \frac{\sin (t\sqrt{-\Delta})}{\sqrt{-\Delta}} g + \int_0^t \frac{\sin((t-\tau)\sqrt{-\Delta})}{\sqrt{-\Delta}} F_p(u)\, d\tau.
    \end{equation*}
By Euler's formula, it is rewritten as
    \begin{equation*}
    \begin{aligned}
    \Phi(u)
    &= \frac{1}{2}\big( e^{it\sqrt{-\Delta}}+e^{-it\sqrt{-\Delta}} \big) f
    + \frac{1}{2}\Big( \frac{e^{it\sqrt{-\Delta}} - e^{-it\sqrt{-\Delta}} }{\sqrt{-\Delta}} \Big) g \\
    &\qquad + \frac{1}{2}\int_0^t
    \frac{e^{i(t-\tau)\sqrt{-\Delta}}-e^{-i(t-\tau)\sqrt{-\Delta}}}{\sqrt{-\Delta}}  F_p(u)\, d\tau.
    \end{aligned}
    \end{equation*}

For $0<s\leq1/2$ and suitable values of $T, A>0$, it suffices to prove that $\Phi$ defines a contraction map on  
    \begin{align*}
    X(T, A)=  \bigg \lbrace \langle \partial_y \rangle^s u \in C(I; \dot{H}_{x,y}^{\frac{1}{2}} ) \cap L_t^q (I ;L_x^{r} &L_y^{ \widetilde{r}}): \\
    &\sup_{(q, r, \widetilde{r}) \in \mathbb{A}_{W}(N, 1)}\| \langle \partial_y\rangle^s u\|_{L_t^q (I ;L_x^{r} L_y^{ \widetilde{r}})} \leq A \bigg\rbrace
    \end{align*}
equipped with the distance
    $$
    d(u, v)= \sup_{(q, r, \widetilde{r}) \in \mathbb{A}_{W}(N, 1)} \|u-v\|_{L_t^q (I ;L_x^{r} L_y^{\widetilde{r}})}.
    $$

We begin with showing $\Phi(u) \in X$ for $u\in X$.
Using Plancherel's theorem, we see
    \begin{align*}
    \sup_{t\in I}\|\langle \partial_y \rangle^s \Phi(u)\|_{\dot{H}_{x,y}^{\frac{1}{2}}}
    \leq C\| \langle \partial_y \rangle^s f\|_{\dot{H}^{\frac{1}{2}}} &+ C\|  \langle \partial_y \rangle^s g\|_{\dot{H}^{-\frac{1}{2}}}\\
    &+C \sup_{t\in I} \bigg\| \langle \partial_y \rangle^s  \int_0^t \frac{e^{i\tau\sqrt{-\Delta}}}{|\nabla|^{1/2}} F_p(u) \, d\tau\bigg\|_{L^2}.
    \end{align*}
By a dual version of \eqref{hstr_Wav} and applying \eqref{no1-w} in Lemma \ref{lem2}, the last term in the right side is bounded as
 \begin{align*}
\sup_{t\in I} \bigg\| \langle \partial_y \rangle^s  \int_0^t \frac{e^{i\tau\sqrt{-\Delta}}}{|\nabla|^{1/2}} F_p(u) \, d\tau\bigg\|_{L^2}
	&\leq C\| \langle \partial_y \rangle^s F_p(u)\|_{L^{q_1'} (I ;L_x^{r_1'} L_y^{\widetilde{r}_1'})}\\
	 &\leq CT^{\beta_1(N,p,s)}  \| u\|_{L^{q_1}(I;L_x^{r_1} W_y^{s,\widetilde{r}_1})}^p
\end{align*}
for some $(q_1,r_1,\widetilde{r}_1)\in \mathbb{A}_W(N,1)$ for which Lemma \ref{lem2} holds.
Hence,
\begin{align*}
\sup_{t\in I}\| \langle \partial_y \rangle^s  \Phi(u)\|_{\dot{H}^{\frac{1}{2}}}
&\leq C\|  \langle \partial_y \rangle^s f\|_{\dot{H}^{\frac{1}{2}}} + C\|  \langle \partial_y \rangle^s g\|_{\dot{H}^{-\frac{1}{2}}} + CT^{\beta_1(N,p,s)} A^{p}
\end{align*}
for $u\in X$.
Similarly, using \eqref{hstr_Wav}, \eqref{istr_Wav} and \eqref{no1-w} in Lemma \ref{lem2}, we also see
\begin{align*}
\|\Phi(u)\|_{L_t^{q}(I; L_x^{r} W_y^{s, \widetilde{r}})}
&\leq C\|  \langle \partial_y \rangle^s f\|_{\dot{H}^{\frac{1}{2}}} + C\|  \langle \partial_y \rangle^s g\|_{\dot{H}^{-\frac{1}{2}}} + C\| \langle \partial_y \rangle^s F_p(u)\|_{L^{q_1'} (I ;L_x^{r_1'} L_y^{\widetilde{r}_1'})} \nonumber\\
&\leq C\|  \langle \partial_y \rangle^s f\|_{\dot{H}^{\frac{1}{2}}} + C\|  \langle \partial_y \rangle^s g\|_{\dot{H}^{-\frac{1}{2}}} +CT^{\beta_1(N,p,s)} A^{p}
\end{align*}
for $u\in X$.
Consequently, $\Phi(u)\in X$ if
\begin{equation}\label{wne}
C\|  \langle \partial_y \rangle^s f\|_{\dot{H}^{\frac{1}{2}}} + C\|  \langle \partial_y \rangle^s g\|_{\dot{H}^{-\frac{1}{2}}} +CT^{\beta_1(N,p,s)} A^{p} \leq A.
\end{equation}

On the other hand, by \eqref{istr_Wav} and \eqref{no2-w} in Lemma \ref{lem2},
\begin{align*}
d\left( \Phi(u), \Phi(v) \right)
&= d\left(\int_0^t \frac{e^{i(t-\tau)\sqrt{-\Delta}}}{\sqrt{-\Delta}} F_p(u)\, d\tau ,\int_0^t \frac{e^{i(t-\tau)\sqrt{-\Delta}}}{\sqrt{-\Delta}}  F_p(v)\, d\tau  \right) \nonumber\\
& \leq  C\| F_p(u)-F_p(v)\|_{L^{q_1'} (I ;L_x^{r_1'} L_y^{\widetilde{r}_1'})}\nonumber\\
& \leq  2CT^{\beta_1(N,p,s)} A^{p-1} \, d(u, v)
\end{align*}
for $u,v \in X$.
By choosing $A=2C\|  \langle \partial_y \rangle^s f\|_{\dot{H}^{1/2}} + 2C\|  \langle \partial_y \rangle^s g\|_{\dot{H}^{-1/2}}$ and taking $T$ sufficiently small so that
$$2CT^{\beta_1(N,p,s)}A^{p-1} \leq \frac{1}{2},$$
we see that \eqref{wne} holds and 
$$d(\phi(u), \phi(v)) \leq  \frac12d(u, v).$$

\end{document}